\documentclass{amsart}
\usepackage{amsmath,amsfonts,amsthm}
\usepackage{amssymb,latexsym}
\usepackage{graphics}
\usepackage[colorlinks]{hyperref}

\usepackage{amssymb}

\theoremstyle{plain}
\theoremstyle{definition}
\newtheorem{theorem}{Theorem}[section]
\newtheorem{lemma}[theorem]{Lemma}

\newtheorem{corollary}[theorem]{Corollary}

\newtheorem{remark}[theorem]{Remark}
\theoremstyle{remark}

\numberwithin{equation}{section}

\title{Uniformizable functional Alexandroff spaces}
\author[F. Ayatollah Zadeh Shirazi, E. Hakimi, A. Hosseini, R. Rezavand]{Fatemah Ayatollah Zadeh Shirazi, Elahe Hakimi \\ Arezoo Hosseini, Reza Rezavand}
\begin{document}
\begin{abstract}
In the following text we show that the Alexandroff space $X$ is uniformizable  if and only if 
the collection of all smallest neighbourhoods is a partition of $X$. Moreover the
Alexandroff space $X$ is uniformizable and functional Alexandroff ($k-$primal)  if and only if 
the collection of all smallest neighbourhoods  is a partition of $X$ into its finite subsets. 
\end{abstract}
\maketitle
\noindent {\small {\bf 2020 Mathematics Subject Classification:}  54C05, 54A05,  54E15  \\
{\bf Keywords:}}  Alexandroff space, Functional Alexandroff space, ($k-$)Primal space, Topological group, Uniformizable.
\section{Introduction}
\noindent Diskrete R\"aume~\cite{alexandroff}, $A-$space~\cite{cord}, 
Alexandroff space~\cite{aren}, etc., all of them denote topological spaces whose topology is closed under arbitrary nonempty intersections.
Amongst subcategories of the category of Alexandroff spaces 
are finite topological spaces (see e.g.~\cite{herda, stong}) that are most studied ones. Another subcategory of the category of Alexandroff spaces 
is the category of 
functional Alexandroff spaces which is introduced in~\cite{man-golestani} on the base of a talk~\cite{talk} which is called and introduced independently
as primal spaces in~\cite{echi}. Moreover $k-$primal spaces studied for the first time in~\cite{sami}, satisfy the following
diagram for a topological space $X$ (hence the category of $k-$primal spaces is an intermediate category):
\begin{center}
$X$ is functional Alexandroff $\Rightarrow$ $X$ is $k-$primal $\Rightarrow$ $X$ is Alexandroff 
\end{center}
In this text we prove that for uniformizable topological space $X$, all of the statements in the above diagram are equivalent.
In topological space $X$, for $a\in X$ and $f:X\to X$ let 
\begin{center}
$V(a):=\bigcap\{U:U$ is an open neighbourhood of $a\}$,
\\ $\:$ \\
$V_f(a):=\{f^{-n}(a):n\geq0\}$. \\ $\:$
\end{center}
So the topological space $X$ is an Alexandroff
space, if $V(a)$ is open for each $a\in X$, i.e. $\{V(a):a\in X\}$ is a basis for the topology of $X$.
\\
For self--map $f:X\to X$, the Alexandroff topology on $X$ generated by basis $\{V_f(a):a\in X\}$
is called the functional Alexandroff topology on $X$ (corresponding to $f$), in this case $V_f(a)=V(a)$ is the smallest open neighbourhood of $a$, for each $a\in X$. We say
the topological space $X$ is a functional Alexandroff space if there exists $f:X\to X$ such that functional Alexandroff topology on $X$ corresponding to $f$
is compatible with the original topology of $X$~\cite{man-golestani}.
\\
Moreover  the topological space $X$ is a $k-$primal space (for $k\geq1$) if there exist $f_1,\ldots,f_k:X\to X$
such that $V(a)=V_{f_1}(a)\cap\cdots\cap V_{f_k}(a)$ is the smallest open neighbourhood of $a$ for each $a\in X$. In particular
any functional Alexandroff space is a $1-$primal space and any $k-$primal space is a $k+1-$primal space~ \cite{sami}.
\subsection*{Background on uniform spaces} For arbitrary set $X$, the nonempty collection $\mathcal F$ of subsets of $X\times X$
is called a uniform structure on $X$ if for each $\alpha\in\mathcal{F}$ we have \cite{du}:
\begin{itemize}
\item $\Delta_X\subseteq\alpha$ (where $\Delta_X:=\{(x,x):x\in X\}$),
\item if $\alpha\subseteq\beta\subseteq X\times X$, then $\beta\in\mathcal{F}$,
\item $\alpha\cap\beta\in\mathcal{F}$, for each $\beta\in\mathcal{F}$,
\item $\alpha^{-1}=\{(y.x):(x,y)\in\alpha\}\in\mathcal{F}$, 
\item there exists $\beta\in\mathcal{F}$ such that $\beta\circ\beta=\{(x,z):\exists y\:\:((x,y)\in\beta\wedge(y,z)\in\beta)\}\subseteq\alpha$.
\end{itemize}
If $\mathcal F$ is a uniform structure on $X$, then 
$\tau_{\mathcal{F}}:=\{U\subseteq X:\forall x\in U\exists\alpha\in\mathcal{F}\:\alpha[x]\subseteq U\}$ is a topology on $X$
(where $\alpha[x]=\{y:(x,y)\in\alpha\}$ for $\alpha\in\mathcal{F},x\in X$). We call $\tau_{\mathcal F}$, the uniform topology
on $X$ generated by $\mathcal F$. We call $(X,\mathcal{F})$, a uniform space and equip it with topology $\tau_{\mathcal F}$.
\\
A topological space $(Y,\tau)$ is uniformizable if there exists a uniform structure $\mathcal H$ on $Y$ such that $\tau=\tau_{\mathcal H}$, in
this case we say $\mathcal H$ is a compatible uniform structure on topological space $Y$.
\begin{remark}
Topological space $X$ is uniformizable if and only if it is generated with a collection of pseudometrics on $X$, i.e.
there exists a nonempty collection $\{p_\theta:\theta\in T\}$ of pseudometrics on $X$ such that
$\{B^{p_\theta}(x,r):x\in X.r>9,\theta\in T\}$ is a sub--base of $X$ (where $B^p(x,r)=\{y\in X:p(x,y)<r\}$ for all $r>0,x\in X$
and pseudometric $p$ on $X$)~\cite{engel}. 
\end{remark}
\section{Main Theorems}
\noindent In this section we charachterize ``uniformizable Alexandroff spaces'' and ``uniformizable functional Alexandroff ($k-$primal) spaces''.
In particular, for nonempty set $X$, there exists a one--to--one correspondence between the family of all uniformizable Alexandroff  topologies on $X$ and
the family of all partitions of $X$. Also there exists a one--to--one correspondence between the family of all uniformizable functional Alexandroff  topologies on $X$ and
the family of all partitions of $X$ into its finite subsets.
\begin{lemma}\label{salam10}
In uniform topological space $X$, for $a\in X$ if $V(a)$ is open, then for each $b\in V(a)$, $V(a)$, too, is the smallest
open neighbourhood of $b$.
\end{lemma}
\begin{proof}
Suppose that $a\in X$ has the smallest open neighbourhood $V(a)$ and let $b\in V(a)\setminus\{a\}$, 
it's clear that $V(a)$ is an open neighbourhood 
of $b$. Consider the uniform topology on $X$ generated by nonempty family $\{p_\alpha:\alpha\in\Gamma\}$ of pseudometrics
on $X$. For each $\theta\in\Gamma$ and $r>0$, $V(a)\cap B^{p_\theta}(a,r)$ is
an open neighbourhood of $a$, hence $b\in V(a)\subseteq V(a)\cap B^{p_\theta}(a,r)\subseteq B^{p_\theta}(a,r)$, thus
$p_\theta(a,b)<r$. 
Hence $p_\theta(a,b)=0$ for each $\theta\in\Gamma$. Thus
$a\in V$, for each open neighbourhood of $b$. Therefore $V(a)\subseteq V$ for each open neighbourhood
$V$ of $b$, which completes the proof.
\end{proof}
\begin{lemma}\label{salam20}
If $\mathcal P$ is a partition of $X$, and $\alpha_0:=\bigcup\{D\times D:D\in\mathcal{P}\}$, then
\begin{itemize}
\item[1.] $\mathcal{F}_{\mathcal P}:=\{\alpha\subseteq X\times X:\alpha_0\subseteq\alpha\}$ is a uniform structure on $X$, 
\item[2.] uniform space $(X,\mathcal{F}_{\mathcal P})$ is an Alexandroff topological space,
\item[3.] in uniform Alexandroff topological space $(X,\mathcal{F}_{\mathcal P})$, $\mathcal{P}=\{V(x):x\in X\}$.
\end{itemize}
\end{lemma}
\begin{proof}
(1) Consider $\alpha,\beta\in\mathcal{F}_{\mathcal P}$ and $\gamma\subseteq X\times X$, then:
\begin{itemize}
\item $\Delta_X\subseteq\alpha_0\subseteq\alpha$,
\item $\alpha_0\subseteq\alpha$ and $\alpha_0\subseteq\beta$ implies $\alpha_0\subseteq\alpha\cap\beta$, thus
	$\alpha\cap\beta\in \mathcal{F}_{\mathcal P}$,
\item $\alpha_0\subseteq\alpha$ and $\alpha\subseteq\gamma$ implies $\alpha_0\subseteq\gamma$, thus 
	$\gamma\in \mathcal{F}_{\mathcal P}$,
\item $\alpha_0\subseteq\alpha$ implies $\alpha_0=\alpha_0^{-1}\subseteq\alpha^{-1}$, thus 
	$\alpha^{-1}\in \mathcal{F}_{\mathcal P}$,
\item if $(x,y),(y,z)\in\alpha_0$, then there exists $E,F\in\mathcal{P}$ such that $(x,y)\in E\times E$ and $(y,z)\in F\times F$.
	Thus $y\in E\cap F$, which leads to $E=F$ (since $E,F$ are elements of partition $\mathcal P$). Hence
	$(x,z)\in E\times F=E\times E\subseteq\alpha_0$. The above discussion shows $\alpha_0\circ\alpha_0\subseteq\alpha_0
	\subseteq\alpha$, in particular, there exists $(\alpha_0=)\theta\in \mathcal{F}_{\mathcal P}$ with
	$\theta\circ\theta\subseteq\alpha$.
\end{itemize} 
Hence $\mathcal{F}_{\mathcal P}$ is a uniform structure on $X$.
\\
(2, 3) For each $x\in X$, there exists unique $D_x\in{\mathcal P}$ such that $x\in D_x$. For each $y\in X$
and $\alpha\in \mathcal{F}_{\mathcal P}$ we have:
\[\alpha[y]\supseteq\alpha_0[y]=D_y\:.\]
Hence $U$ is an open subset of  $(X,\mathcal{F}_{\mathcal P})$ if and only if $D_y\subseteq U$ for each $y\in U$.
Let $x\in X$, for each $y\in D_x$ we have $D_x=D_y$ thus $D_x=\bigcup\{D_y:y\in D_x\}=\bigcup\{\alpha_0[y]:y\in D_x\}$ is an open subset of $(X,\mathcal{F}_{\mathcal P})$. On the other hand, if $U$ is an open neighbourhood of
$X$, then there exists $\alpha\in \mathcal{F}_{\mathcal P}$ such that $D_x=\alpha_0[x]\subseteq\alpha[x]\subseteq U$,
hence $D_x$ is the smallest open neighbourhood of $x$. In particular $(X,\mathcal{F}_{\mathcal P})$ is an 
Alexandroff space (since each point has the smallest open neighbourhood).
\end{proof}
\begin{lemma}\label{salam22}
If $X$ is a  uniformizable Alexandroff space, then
$\{V(a):a\in X\}$ is a partition of $X$ (into its open subsets).
\end{lemma}
\begin{proof}
Suppose $X$ is a uniformizable Alexandroff  space. It's clear that $\{V(a):a\in X\}$ is a collection of nonempty subsets of $X$ with
	$\bigcup\{V(a):a\in X\}=X$ (since $x\in V(x)\subseteq X$ for each $x\in X$). For $x,y\in X$ with 
	$V(x)\cap V(y)\neq\varnothing$, choose $z\in V(x)\cap V(y)$, by Lemma~\ref{salam10} the equality $V(x)=V(z)=V(y)$ is valid.
\end{proof}
\begin{theorem}\label{salam30}
In topological space $X$ the following statements are equivalent:
\begin{itemize}
\item[1.] $X$ is a  uniformizable Alexandroff space,
\item[2.] $\{V(a):a\in X\}$ is a partition of $X$ into its open subsets,
\item[3.] There exists a partition $\mathcal P$ of $X$ such that  $\mathcal{F}_{\mathcal P}:=\{\alpha\subseteq X\times X:\bigcup\{D\times D:D\in{\mathcal P}\}\subseteq\alpha\}$ is a compatible uniform structure on $X$,
\item[4.] $X$ is an Alexandroff space and
\[\forall a\in X\:\:\forall b\in V(a)\:\:(V(b)=V(a))\:,\]
\item[5.] $X$ is an Alexandroff space and $\{(a,b)\in X\times X:V(b)\subseteq V(a)\}$ is an equivalence relation on $X$,
\item[6.] quotient space
$\frac{X}{\Re}$ is a discrete space, where $\Re:=\{(a,b)\in X\times X:V(b)=V(a)\}$ (consider natural quotient map
$\pi:X\to\frac{X}{\Re}$ ($\pi(x)=\frac{x}{\Re}$)).
\end{itemize}
\end{theorem}
\begin{proof}  (1), (2) and (3) are equivalent by Lemmas~\ref{salam20} and~\ref{salam22}. (2), (4) and (5) are clearly equivalent using the difinition 
of a partition of $X$.
\\
($2\Rightarrow6$): Suppose $\{V(a):a\in X\}$ is a partition of $X$ into its open subsets, then for each $a\in X$, $\pi^{-1}(\frac{a}{\Re})=V(a)$
is an open subset of $X$. Thus $\{\frac{a}{\Re}\}$ is an open subset of $\frac{X}{\Re}$ for each $a\in X$, i.e.
$\frac{X}{\Re}$ is a discrete space.
\\
($6\Rightarrow4$): Suppose $\frac{X}{\Re}$ is a discrete space and let $a\in X$. $\pi^{-1}(\frac{a}{\Re})$
is an open neighbourhood of $a$ (since $\{\frac{a}{\Re}\}$ is an open subset of $\frac{X}{\Re}$), thus 
$V(a)\subseteq \pi^{-1}(\frac{a}{\Re})$. On the other hand, $\pi^{-1}(\frac{a}{\Re})=\{b\in X:V(a)=V(b)\}\subseteq V(a)$
(since for each $b\in X$, $b\in V(b)$). Hence 
\[\pi^{-1}(\frac{a}{\Re})=V(a)\tag{*}\]
is open and $a$ has the smallest open neighbourhood. Therefore $X$ is an Alexandroff space. 
\\
Moreover for each $a\in X$ and $b\in V(a)$, we have (use ($*$)):
\begin{eqnarray*}
b\in V(a) & \Rightarrow & b\in \pi^{-1}(\frac{a}{\Re}) \\
& \Rightarrow & b\in \frac{a}{\Re} \\
& \Rightarrow & \frac{b}{\Re}=\frac{a}{\Re}  \\
& \Rightarrow & V(b)=\pi^{-1}(\frac{b}{\Re})=\pi^{-1}(\frac{a}{\Re})=V(a)
\end{eqnarray*}
which leads to (4).
\end{proof}
\begin{remark}\label{salam40}
Alexandroff space $X$ is functional Alexandroff if and only if the following conditions hold \cite[Theorem~3.5]{man-golestani}:
\\ {\small
$(\mathcal{C}_1)$ $\forall x,y\in X\:\: (V(x)\cap V(y)=\varnothing\vee V(x)\subseteq V(y)\vee V(y)\subseteq V(x))$,
\\
$(\mathcal{C}_2)$ $\forall x\in X\:\:((\exists y\in X\:\:(V(x){\rm \: is \: a \: proper \: subset \: of \:}V(y)))\Rightarrow(\forall z\in X
	\setminus\{x\}\:\:V(z)\neq V(x)))$,
\\
$(\mathcal{C}_3)$ $\forall x,y\in X\:\:(\{z\in X:V(y)\subseteq V(z)\subseteq V(x)\}{\rm \: is \: finite})$.
}
\end{remark}
\begin{corollary}\label{salam50}
An Alexandroff uniformizable space is a functional Alexandroff uniformizable space if and only if
the smallest open neighbourhood of each point is finite.
\end{corollary}
\begin{proof}
Suppose $X$ is a uniformizable Alexandroff space. By Theorem~\ref{salam30}, $\{V(a):a\in X\}$ is a collection of
disjoint sets, therefore $(\mathcal{C}_1)$ and $(\mathcal{C}_2)$ (in Remark~\ref{salam40}) are valid. Hence $X$ is a functional Alexandroff space
if and only if $(\mathcal{C}_3)$. Since $\{V(a):a\in X\}$ is a collection of disjoint sets and for each $a\in X$, $a\in V(a)$, so:
\[\forall x,y\in X\:\:(V(y)\subseteq V(x)\Leftrightarrow V(y)=V(x))\:,\]
Thus for uniformizable Alexandroff space $X$, $(\mathcal{C}_3)$ is converted to:
\[\forall x\in X\:\:(\{z\in X:V(z)= V(x)\}{\rm \: is \: finite})\]
and by (4) in Theorem~\ref{salam30}:
\[\forall x\in X\:\:(V(x){\rm \: is \: finite})\:,\]
which completes the proof.
\end{proof}
\begin{corollary}\label{salam60}
By Theorem~\ref{salam30}, in finite nonempty $n-$set $X$, the number of all uniformizable topologies 
(resp. Functional Alexandroff uniformizable topologies)
is equal to the number of partiotions of $X$, i.e. Bell number $B_n$, where $B_0:=B_1=1$ and $B_{k+1}=\mathop{\Sigma}\limits_{0\leq i\leq k}
\left(\begin{array}{c} k \\ i \end{array}\right) B_i$ ($k\geq1$).
\end{corollary}
\begin{lemma}\label{salam65}
Every $k-$primal uniformizable space is a functional Alexandroff space.
\end{lemma}
\begin{proof}
Suppose $X$ is a $k-$primal uniformizable space, then $X$ is a uniformizable Alexandroff space and by Theorem~\ref{salam30},
$\{V(a):a\in X\}$ is a partition of $X$ into its minimal open subsets. $X$ is a $k-$prinal space, thus there are $f_1,\ldots,f_k:X\to X$ such that
\[\forall a\in X\:\:(V(a)=V_{f_1}(a)\cap\cdots V_{f_k}(a))\: .\]
For each $a\in X$ and $b\in V(a)$ we have $V(b)=V(a)$ (since $\{V(z):z\in X\}$ is a partition of $X$), also 
$b\in V(b)=V(a)=V_{f_1}(a)\cap\cdots V_{f_k}(a)\subseteq V_{f_1}(a)$ and
$a\in V(a)=V(b)=V_{f_1}(b)\cap\cdots V_{f_k}(b)\subseteq V_{f_1}(b)$. Since
$b\in  V_{f_1}(a)$ (resp. $a\in  V_{f_1}(b)$) we have $V_{f_1}(b)\subseteq V_{f_1}(a)$ (resp. $V_{f_1}(a)\subseteq V_{f_1}(b)$).
Thus $V_{f_1}(b)= V_{f_1}(a)$. So:
\[\forall a\in X\forall b\in V(a) \:\:(V_{f_1}(b)= V_{f_1}(a))\:.\tag{+}\]
In functional Alexsandroff topology on $X$ associated to $f_1$ by $(\mathcal{C}_3)$ in Remark~\ref{salam40} we have:
\[\forall x,y\in X\:\:(\{z\in X: V_{f_1}(y)\subseteq V_{f_1}(z)\subseteq V_{f_1}(x)\}{\rm \: is \: finite})\]
which leads to
\[\forall x\in X\:\:(\{z\in X: V_{f_1}(x)\subseteq V_{f_1}(z)\subseteq V_{f_1}(x)\}{\rm \: is \: finite})\]
i.e.,
\[\forall x\in X\:\:(\{z\in X: V_{f_1}(z)=V_{f_1}(x)\}{\rm \: is \: finite})\:. \tag{++}\]
Using $(+)$ and $(++)$, $V(a)$ is finite for each $a\in X$.
Therefore $\{V(a):a\in X\}$ is a partition of $X$ into its finite open subsets. By Corollary~\ref{salam50}, $X$ is a functional Alexandroff space.
\end{proof}
\begin{theorem}\label{salam70}
In topological space $X$ the following statements are equivalent:
\begin{itemize}
\item[1.] $X$ is a functional Alexandroff ($k-$primal) uniformizable space,
\item[2.] $X$ is an Alexandroff uniformizable space such that for each $x\in X$, $V(x)$ is finite,
\item[3.] $\mathcal{P}=\{V(a):a\in X\}$ is a partition of $X$ into its finite open subsets,
\item[4.] There exists a partition $\mathcal P$ of $X$ into its finite subsets such that  $\mathcal{F}_{\mathcal P}:=\{\alpha\subseteq X\times X:\bigcup\{D\times D:D\in{\mathcal P}\}\subseteq\alpha\}$ is a compatible uniform structure on $X$,
\item[5.] There exists self--map $f:X\to X$ such that $Per(f)=X$ and $X$ is a functional Alexandroff space 
	associated to $f$,
\end{itemize}
\end{theorem}
\begin{proof}
Note that each functional Alexandroff space is a $k-$primal space, so by Lemma~\ref{salam65}, a uniformizable space
is a $k-$primal space if and only if it is a functional Alexandroff space.
\\
(1), (2), (3) and (4) are equivalent by Theorem~\ref{salam30} and Corollary~\ref{salam50}.
\\
($5\Rightarrow3$): For self--map $f:X\to X$ suppose $Per(f)=X$. Thus $f:X\to X$ is bijective and for each $x\in X$,
$V(x)=\bigcup\{f^{-n}(x):n\geq0\}=\{x,f(x),\ldots,f^{per(x)}(x)=x\}$ is finite.
\\
($3\Rightarrow5$): Suppose $\mathcal{P}=\{V(a):a\in X\}$ is a partition of $X$ into its finite open subsets. 
For each $D\in\mathcal{P}$ with $n_D$ elements suppose $D=\{a^D_1,\ldots,a^D_{n_D}\}$ and define $f:X\to X$ 
by
\[f(x)=\left\{\begin{array}{lc}a^D_{i+1}, & x\in D\in\mathcal{P}, x= a^D_{i},i\neq n_D\: \\ \\
a^D_1, & x\in D\in\mathcal{P}, x= a^D_{n_D}\:.\end{array}\right.\]
Then $Per(f)=X$ and $\bigcup\{f^{-n}(x):n\geq0\}=\{x,f(x),\ldots,f^{per(x)}(x)=x\}=D$ for $x\in D\in\mathcal{P}$.
Thus ``Alexandroff topology on $X$ with $\mathcal P$ as the collection of its smallest open neighbourhoods'' and 
``functional Alexandroff topology on $X$ associated with $f$'' are compatible.
\end{proof}
\begin{theorem}\label{salam80}
Functional Alexandroff topology on $X$ corresponding to self--map $f:X\to X$ is uniformizable if and only if $Per(f)=X$.
\end{theorem}
\begin{proof}
For $f:X\to X$ with $Per(f)=X$, by Theorem~\ref{salam70}, functional  Alexandroff topology on $X$ corresponding to $f$ is uniformizable.
\\
Now consider $g:X\to X$ such that functional  Alexandroff topology on $X$ corresponding to $g$ is uniformizable. By Theorem~\ref{salam30},
$\{V(x):x\in X\}=\{\bigcup\{g^{-n}(x):n\geq0\}:x\in X\}$ is a partition of $X$. In particular for each $x\in X$,
$\bigcup\{g^{-n}(x):n\geq0\}=\bigcup\{g^{-n}(g(x)):n\geq0\}$ (since $x\in \bigcup\{g^{-n}(x):n\geq0\}\cap \bigcup\{g^{-n}(g(x)):n\geq0\}$)
hence $g(x)\in \bigcup\{g^{-n}(g(x)):n\geq0\}=\bigcup\{g^{-n}(x):n\geq0\}$ and there exists $n\geq0$ such that
$g(x)\in g^{-n}(x)$, i.e. $g^{n+1}(x)=x$, thus $x\in Per(g)$.
\end{proof}
\subsection*{(Functional) Alexandroff topological groups}
Amongst well--known uniform spaces are metric spaces and topological groups~\cite{engel}.
All metrizable Alexandroff spaces are just discrete spaces. In the sequel we see each Alexandroff topological group
(by Alexandroff topological group we mean a topological group which is an Alexandroff space)
is homeomorphic with a product of a discrete and a non-discrete space.
\begin{theorem}\label{zahra10}
In topological group $G$ the following statements are equivalent:
\begin{itemize}
\item[1.] $G$ is an Alexandroff space,
\item[2.] $V(e)$ is open,
\item[3.] there exist a non-discrete topological group $E$ and a discrete topological group $F$ such that $E\times F$ and $G$ are homeomorphic topological spaces,
\item[4.] there exist a non-discrete topological space $E$ and a discrete topological space $F$ such that $E\times F$ and $G$ are homeomorphic topological spaces.
\end{itemize}
\end{theorem}
\begin{proof} Clearly (1) implies (2), also (3) implies (4).
\\
($2\Rightarrow1$): If $V(e)$ is open, then for each $g\in G$, $V(g)=gV(e)$ is open, and by definition, $G$ is an Alexandroff space.
\\
($1\Rightarrow3$): Suppose $G$ is an Alexandroff topological group, then $G$ is a uniformizable Alexandroff space and by Theorem~\ref{salam30}, $\{V(g):g\in G\}$ is a partition of $G$ into its open subsets. For each $g\in G$, $V(g)=gV(e)$ has non-discrete topology, moreover
for each $g,h\in G$ we have:
\begin{eqnarray*}
V(g)= V(h ) & \Leftrightarrow & gV(e)=hV(e) \\
&\Leftrightarrow & h^{-1}g\in V(e)
\end{eqnarray*}
hence for $\Re=\{(g,h)\in G\times G:V(g)=V(h)\}$ we have $\frac{G}{\Re}=\frac{G}{V(e)}$. By Theorem~\ref{salam30},
$\frac{G}{V(e)}$ is a discrete topological group (note that $V(e)$ is a normal subgroup of $G$). 
For each $D\in \{V(g):g\in G\}=\frac{G}{V(e)}$, choose $g_D\in D$ (in particular $D=g_DV(e)$). We claim that
$\varphi:V(e)\times \frac{G}{V(e)}\to G$ with $\varphi(h,D)=g_Dh$ is a homeomorphism.
\\
$\bullet$ $\varphi:V(e)\times \frac{G}{V(e)}\to G$ is one--to--one: Consider $(h_1,D_1),(h_2,D_2)\in V(e)\times \frac{G}{V(e)}$ with $g_{D_1}h_1=g_{D_2}h_2$. Then
\[D_1= g_{D_1}V(e)\mathop{=}\limits^{h_1\in V(e)}g_{D_1}h_1V(e) 
\mathop{=}\limits^{g_{D_1}h_1=g_{D_2}h_2} g_{D_2}h_2V(e) \\
\mathop{=}\limits^{h_2\in V(e)} g_{D_2}V(e)=D_2\]
thus $D_1=D_2$, therefore $g_{D_1}=g_{D_2}$. Using $g_{D_1}=g_{D_2}$ and $g_{D_1}h_1=g_{D_2}h_2$ we have
$h_1=h_2$.
\\
$\bullet$  $\varphi:V(e)\times \frac{G}{V(e)}\to G$ is onto: note that
$\varphi(V(e)\times \frac{G}{V(e)})=\{g_Dh:h\in V(e),D\in \frac{G}{V(e)}\}=\bigcup\{g_D V(e):D\in \frac{G}{V(e)}\}=\bigcup\{D:D\in \frac{G}{V(e)}\}=\bigcup\frac{G}{V(e)}=G$.
\\
$\bullet$  $\varphi:V(e)\times \frac{G}{V(e)}\to G$ is continuous: Consider $(h,D)\in V(e)\times \frac{G}{V(e)}$ and
suppose $U$ is an open neighbourhood of $g_D h$ in $G$, thus $U\supseteq V(g_Dh)=g_DhV(e)=g_DV(e)=\varphi(V(e)\times\{D\})$. $V(e)\times\{D\}$ is an open neighbourhood of $(h,D)$
which leads to continuity of $\varphi:V(e)\times \frac{G}{V(e)}\to G$ in $(h,D)$.
\\
$\bullet$  $\varphi:V(e)\times \frac{G}{V(e)}\to G$ is an open map: the smallest open neighbourhood of each $(h,D)\in V(e)\times \frac{G}{V(e)}$ is $V(e)\times\{D\}$. Now suppose $W$ is an open subset of $V(e)\times \frac{G}{V(e)}$, then
\begin{eqnarray*}
\varphi(W) & = & \varphi(\bigcup\{V(e)\times\{D\}:D\in \frac{G}{V(e)}\wedge(\exists h\:\: (h,D)\in W)\}) \\
& = & \bigcup\{\varphi(V(e)\times\{D\}):D\in \frac{G}{V(e)}\wedge(\exists h\:\: (h,D)\in W)\} \\
& = & \bigcup\{g_DV(e):D\in \frac{G}{V(e)}\wedge(\exists h\:\: (h,D)\in W)\}
\end{eqnarray*}
hence $\varphi(W)$ is an open subset of $G$, since $gV(e)$ is open (for each $g\in G$) by openness of $V(e)$.
\\
($4\Rightarrow1$): Consider non-discrete space $E$, discrete space $F$ and $(x,y)\in E\times F$, then
$E\times\{y\}$ is the smallest open neighbourhood of $(x,y)$. Hence $E\times F$ is an Alexandroff space.
\end{proof}
\begin{theorem}\label{zahra20}
In topological group $G$ the following statements are equivalent:
\begin{itemize}
\item[1.] $G$ is a functional Alexandroff ($k-$primal) space,
\item[2.] $V(e)$ is open and finite,
\item[3.] there exist a non-discrete topological group $E$ and a discrete finite topological group $F$ such that $E\times F$ and $G$ are homeomorphic topological spaces,
\item[4.] there exist a non-discrete topological space $E$ and a discrete finite topological space $F$ such that $E\times F$ and $G$ are homeomorphic topological spaces.
\end{itemize}
\end{theorem}
\begin{proof}
(1) implies (2) by Theorem~\ref{salam70}.
\\
($2\Rightarrow1$) If $V(e)$ is a finite open subset of $G$ with $n(\geq1)$ elements, then for each $g\in G$, $V(g)=gV(e)$ 
 a finite open subset of $G$ with $n(\geq1)$ elements too. For $g,h\in G$, if $V(g)\cap V(h)\neq\varnothing$, then choose $k\in V(g)\cap V(h)$.
 Thus $V(k)$ as the smallest open neighbourhood of $k$ is the subset of $V(g)\cap V(h)$. Since $V(g), V(h), V(k)$ have $n$ elements and
 $V(k)\subseteq V(g),V(k)\subseteq V(h)$, we have $V(g)=V(k)=V(h)$. By Theorem~\ref{salam70}, $G$ is a functional Alexandroff space.
\\
Use a similar method described in Theorem~\ref{zahra10} to complete the proof.
\end{proof}

\noindent \noindent {\small {\bf Fatemah Ayatollah Zadeh Shirazi}, Faculty
of Mathematics, Statistics and Computer Science, College of
Science, University of Tehran, Enghelab Ave., Tehran, Iran
\linebreak (f.a.z.shirazi@ut.ac.ir)}
\\
{\small {\bf Elaheh Hakimi}, Faculty of Mathematics, Statistics
and Computer Science, College of Science, University of Tehran,
Enghelab Ave., Tehran, Iran (elaheh.hakimi@gmail.com)}
\\
{\small {\bf Arezoo Hosseini},
Faculty of Mathematics, College of Science, Farhangian University, Pardis Nasibe--shahid sherafat, Enghelab Ave., Tehran, Iran
(a.hosseini@cfu.ac.ir)}
\\
{\small {\bf Reza Rezavand}, School of Mathematics, Statistics
and Computer Science, College of Science, University of Tehran,
Enghelab Ave., Tehran, Iran (rezavand@ut.ac.ir)}

\end{document}